\documentclass[11pt]{amsart}
\usepackage[utf8]{inputenc}
\usepackage{hyperref}
\usepackage[shortlabels]{enumitem}
\usepackage{graphicx}

\usepackage{thmtools}

\declaretheorem[name=Theorem,numberwithin=section]{theorem}
\declaretheorem[name=Lemma,sibling=theorem]{lemma}
\declaretheorem[name=Proposition,sibling=theorem]{proposition}
\declaretheorem[name=Corollary,sibling=theorem]{corollary}

\declaretheorem[name=Remark,sibling=theorem]{remark}

\declaretheorem[name=Question,sibling=theorem]{ques}

\newcommand{\dd}{\,\mathrm d}
\newcommand{\half}{\frac{1}{2}}
\makeatletter
\newcommand*{\Relbarfill@}{\arrowfill@\Relbar\Relbar\Relbar}
\newcommand*{\xeq}[2][]{\ext@arrow 0055\Relbarfill@{#1}{#2}}
\makeatother

\allowdisplaybreaks

\title{The Broken Stick Project}
\author{P.\ A.\ CrowdMath}
\date{\today}

\begin{document}
\maketitle

\begin{abstract}
The \emph{broken stick problem} is the following classical question.
\begin{quote}
	You have a segment $[0,1]$.
	You choose two points on this segment at random.
	They divide the segment into three smaller segments.
	Show that the probability that the three segments
	form a triangle is $1/4$.
\end{quote}
The MIT PRIMES program, together with Art of Problem Solving,
organized a high school research project where participants
worked on several variations of this problem.
Participants were generally high school students
who posted ideas and progress to the Art of Problem Solving
forums over the course of an entire year,
under the supervision of PRIMES mentors.

This report summarizes the findings of this CrowdMath project.
\end{abstract}


\section{Introduction}
The \emph{broken stick problem} is the following classical question.
\begin{quote}
	You have a segment $[0,1]$.
	You choose two points on this segment at random.
	They divide the segment into three smaller segments.
	Show that the probability that the three segments
	form a triangle is $1/4$.
\end{quote}
See \cite{orig} for one possible reference.
This is also called the ``broken spaghetti problem'',
due to a popular fact that spaghetti noodles almost
never break in half when force is applied on the two ends
\cite{spaghetti}.

The 2017 entrance exam for MIT PRIMES \cite{mitprimes}
asked applicants to come up
with interesting generalizations and variations of the problem.
PRIMES decided to use some of these proposed variations
to start a polymath project hosted on the
Art of Problem Solving forums \cite{aops}.
The project was done under the name P.\ A.\ Crowdmath.
The students worked on some of these variations they
themselves suggested, and came up with new directions.
The project was guided by PRIMES mentors Evan Chen
and Dr.~Tanya Khovanova.

The main results obtained broadly fall into three categories.
In \S\ref{sec:known}, we give results about the expected area and
the length of the $k$th longest segment
if the interval $[0,1]$ is randomly divided into $n$ segments.
In \S\ref{sec:square}, we consider an analogous problem where a unit square $[0,1]^2$
is broken into pieces.
Finally, in \S\ref{sec:numtri} we give results about the probability
that $m$ or more $k$-gons are formed if $[0,1]$ is split into $n$ pieces.

\section{Expected area and $k$th moment}
\label{sec:known}
We begin with a couple variations of the problem
which were found in the literature.

\subsection{Expected area of triangle}
One of the most natural geometric questions about the formed triangle
is the area, answered in e.g.\ \cite{gbook1}.
\begin{proposition}
	[{\cite[page 269-270]{gbook1}}]
	\label{areatriangle}
	We break a unit segment randomly into three segments.
	Given that the three segments form a triangle,
	the expected area is $\frac{\pi}{105}$.
\end{proposition}
%
%
Similarly, for a stick of length $k$,
the expected value of the area of the triangle is $\frac{\pi}{105}k^2$.

Here is a variant:
\begin{proposition}
	Suppose we split a unit segment into 3 pieces.
	If we are unable to form a triangle,
	we take the largest segment and split it into 3 pieces again.
	With these 3 pieces, we try to create a triangle.
	If we fail, we continue the process.

	At the end of the process,
	the expected perimeter is $1/2$ and the expected area
	is $\frac{8\pi}{2205}$.
\end{proposition}
\begin{proof}
	Let us solve this with events with states. There is a probability of $\frac{1}{4}$ of being able to form a triangle and a probability of $\frac{3}{4}$ of not being able to form a triangle.

	If we do not form a triangle, then without the loss of generality, suppose both breaks are to the left of $\frac{1}{2}$. They divide the interval $[0,\frac{1}{2}]$ into three pieces. Let the leftmost piece be $a$, middle be $b$ and rightmost be $c$. So we have $a+b+c=\frac{1}{2}$. By symmetry, the expected value of $a, b,$ and $c$ are $\frac{1}{6}$, so the expected value of $a+b$ is $\frac{1}{3}$. After removing $a$ and $b$, we are left with the largest segment (a piece with length $\frac{2}{3}$) of the triangle and start over.

	Let $P$ be the expected perimeter. Then $P=\frac{1}{4} \times1 + \frac{3}{4} \times \frac{2}{3}P$. Solving, we get $P=\frac{1}{2}$.

	If the perimeter is reduced to $X$, then the area is multiplied by a factor of $X^2$, so we should find the expected value of $(c+\frac{1}{2})^2$. Note that $\mathbb E[(c+\frac{1}{2})^2] = \mathbb E[c^2+ c + \frac{1}{4}] = \mathbb E[c^2] + \mathbb E[c] + \frac{1}{4}$. We know $\mathbb E[c] = \frac{1}{6}$, so it remains to find $\mathbb E[c^2]$.

	\[ \frac{\int\limits_{0}^{\frac{1}{2}} \int\limits_{x}^{\frac{1}{2}} x^2\; dy\; dx}{ \int\limits_{0}^{\frac{1}{2}} \int\limits_{x}^{\frac{1}{2}} 1\; dy\; dx} = \frac{1}{24}. \]

	So $\mathbb E[(c+\frac{1}{2})^2] =\frac{11}{24}$

	Let $A$ be the expected area. Then $A= \frac{1}{4} \times \frac{\pi}{105} + \frac{3}{4} \times \frac{11}{24}A$. Solving, we get $A=\frac{8\pi}{2205}$.
\end{proof}

We also have the following result
about whether the area formed can exceed a certain value.
\begin{theorem}
	Let $A_0 \in (0, \sqrt{3}/36)$.
	Given that a triangle can be formed,
	the probability that its area $A$ exceeds $A_0$ is given by
	\[ \mathbb P(A > A_0) = 4\int_{\mu_1}^{\mu_2} \sqrt{k^2(1-k^2)^2 - (8A_0)^2} \; dk \]
	where $0 < \mu_1 < \mu_2 < 1$ are the two roots of
	the polynomial $k(1-k^2) - 8A_0$.
\end{theorem}
Note that the $\sqrt{3}/36$
corresponds to the equilateral triangle with side length $1/3$,
the largest area of a triangle with perimeter $1$.

\begin{proof}
	Define the three sides of the triangle as $a,b,c$
	where $a$ is the distance from the endpoint to $x$,
	$b$ is the distance from $x$ to $y$,
	and $c$ is the distance from $y$ to the other endpoint.
	We can say then that a triangle only forms when 
	\[ a \le \frac12,\;c \le \frac12,\;a+b \ge \frac12. \]
	Define $u=1-2a$ and $v=1-2b$,
	where a triangle is formed once again where $a$ and $b$ are $<0.5$
	and their sum is also greater than $0.5$,
	we already know this occurs with probability $0.25$.

	Now we go to Heron's formula
	\[ A = \sqrt{s(s-a)(s-b)(s-c)} \]
	which can be factorized using our new definitions of $u$ and $v$ as
	\[ 64A^2 = ((x+y)^2 - (x-y)^2)(1-x-y). \]
	Create the new variables $t=(u+v)$ and $w=(u-v)$,
	and note that \[ A \ge A_0 \iff w^2 \le p^2 - \frac{(8A_0)^2}{1-t}. \]
	We then define
	\[ f(t) = \sqrt{t^2 - \frac{(8A_0)^2}{1-t}} \]
	and let $\pi_1$, $\pi_2$ be the two roots of $f(p)$ in $(0,1)$.
	The condition above is equivalent to $\pi_1 \le t \le \pi_2$
	and $|w| \le f(p)$.
	Noting that $dt \; dw = 2du \; dv$, we conclude
	\[ \mathbb P(A) = \int_{\pi_1}^{\pi_2} \int_{-f(t)}^{f(w)} \; dt \; dw
		= 2\int_{\pi_1}^{\pi_2} f(t) \; dt. \]
	Change variables to $k = \sqrt{1-t}$,
	this becomes
	\[ \mathbb P(A>A_0)
		= 4\int_{\mu_1}^{\mu_2} \sqrt{k^2(1-k^2)^2 - (8A_0)^2} \; dk \]
	where $\mu_1$, $\mu_2$ are the roots of the polynomial $k(1-k^2) - 8A_c$.
\end{proof}

\subsection{Areas of $n$-gons}
Now suppose instead we break into $n \ge 4$ segments.
If these segments form an $n$-gon,
we can consider the $n$-gon with maximum area that
can be formed, and ask for its expected value.
This is equivalent to finding the area of the cyclic $n$-gon.

\begin{lemma}
	[\cite{maxquad}]
	\label{cyclicngon}
	For $n \ge 4$, the largest $n$-gon with
	a specified set of sides is cyclic.
\end{lemma}
\begin{proof}
	We first prove the result when $n = 4$.
	By Bretschneider's formula, the area of a quadrilateral with side lengths $a, b, c, d$, semiperimeter $s$ and opposite angles $\alpha$ and $\gamma$ is
	\[ \sqrt{(s-a)(s-b)(s-c)(s-d)-abcd\cos^2\left(\frac{\alpha+\gamma}{2}\right)}. \]
    Since $a,b,c,d$ and thus also $s$ are fixed, we only need to minimize $\cos^2\left(\frac{\alpha+\gamma}{2}\right)$ to maximize the area. The minimum is achieved when $\cos^2\left(\frac{\alpha+\gamma}{2}\right)=0$, or equivalently, $\alpha + \gamma = \pi$, which implies that the quadrilateral is cyclic.

	Now assume $n > 4$.
	Assume for the sake of contradiction that there was a non-cyclic polygon that had the optimal area with fixed side lengths. Obviously it would be convex.
	\begin{center}
		\includegraphics{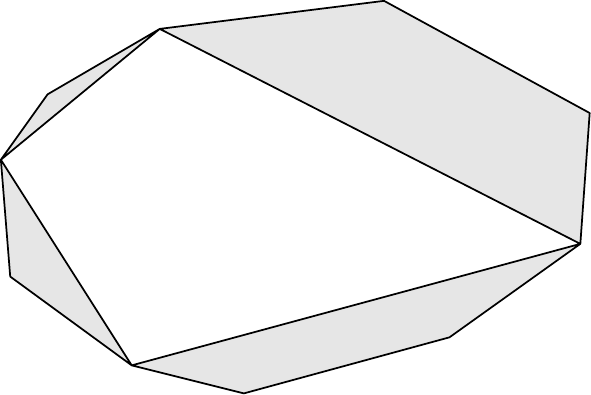}
	\end{center}
	Then, some quadrilateral which has vertices from the polygon would not be cyclic. Then, fixing the other parts of the polygon (corresponding to the shaded regions in the diagram) while adjusting the quadrilateral by fixing the side lengths and making it cyclic would increase the total area of the polygon --- a contradiction. Therefore, the claim is true.
\end{proof}

\begin{proposition}
	For $n=4$, the expected area of the cyclic quadrilateral
	(with maximal area)
	formed is $\frac{17\pi}{525}-\frac{\pi ^2}{160}.$
\end{proposition}

\begin{proof}
Let $x<y<z$ be the breaking points.
Then, the quadrilateral has side lengths $x, y-x, z-y,$ and $1-y$.
By symmetry, we can assume $y < \frac 1 2$.
Since the perimeter of the quadrilateral is $1$, the semiperimeter is $\frac 1 2$.

Since a quadrilateral is valid iff all side lengths are shorter than the semiperimeter,
we have $x < \frac 1 2, y-x < \frac 1 2, z-y < \frac 1 2$ and $z > \frac 1 2.$ 
By proposition \ref{cyclicngon}, the quadrilateral with maximal area is cyclic.

By Brahmagupta's formula, the expected value of the
area of this quadrilateral is then
\[ 
A = \frac{\int_{0}^{1/2} \int_0^y \int_{1/2}^{y+1/2} \sqrt{\left(\frac 12 - x\right)\left(\frac 12-y+x\right)\left(\frac 12-z+y\right)\left(z-\frac 12\right)} \dd z \dd x \dd y}{\int_{0}^{1/2} \int_0^y \int_{1/2}^{y+1/2} 1 \dd z \dd x \dd y}. \]

We get
\[ A = \frac{ \frac{17\pi}{12600}-\frac{\pi^2}{3840}}{\frac 1 {24}}
=  \frac{17\pi}{525}-\frac{\pi^2}{160}. \]

Evaluation of the integral:
    \begin{align*}
    \int_{0}^{1/2} &\int_0^y \int_{1/2}^{y+1/2} \sqrt{\left(\frac 12 - x\right)\left(\frac 12-y+x\right)\left(\frac 12-z+y\right)\left(z-\frac 12\right)} \dd z \dd x \dd y = \\
    &= \frac 14\int_0^{1/2} \int_0^y \sqrt{(1-2x)(2x-2y+1)} \int_{1/2}^{y+1/2} \sqrt{y^2-\left( 2z-y-1\right)^2} \dd z \dd x \dd y\xeq{u=2z-y-1} \\
    &= \frac 18 \int_0^{1/2}\int_0^y \sqrt{(1-2x)(2x-2y+1)} \int_{-y}^{y} \sqrt{y^2-u^2} \dd z \dd x \dd y\xeq{t=\arcsin\left(\frac uy\right)} \\
    &= \frac\pi{16}\int_0^{1/2} y^2\int_0^y \sqrt{(1-2x)(2x-2y+1)} \dd x \dd y = \\
    &= \frac\pi{16}\int_0^{1/2} y^2\int_0^y \sqrt{(y-1)^2 - (2x-y)^2} \dd x \dd y \xeq{u=2x-y} \\
    &= \frac{\pi}{32} \int_0^{1/2} y^2 \int_{-y}^y \sqrt{(y-1)^2-u^2} \dd u \dd y \xeq{v=\arcsin\left( \frac{u}{1-y}\right)} \\
    &= \frac{\pi}{32} \int_0^{1/2} y^2(y-1)^2 \int_{-\arcsin\left(\frac{y}{1-y}\right)}^{\arcsin\left(\frac{y}{1-y}\right)} \cos ^2 v \dd v \dd y = \\
    &= \frac{\pi}{32}\int_0^{1/2} y^2 (y-1)^2 \arcsin\left( \frac{y}{1-y}\right) \dd y + \frac{\pi}{32} \int_0^{1/2} y^3 \sqrt{1-2y} \dd y\xeq{w=y-1, a=\sqrt{-2w-1}} \\
    &= -\frac{\pi}{120} - \frac 1 {240} \int_0^{1/2} \frac{3a^{10}-10a^6+15a^2-8}{a^2+1} \dd a + \frac{\pi}{32} \int_0^{1/2} y^3 \sqrt{1-2y} \dd y\xeq{b=1-2y} \\
	&= \frac{17\pi}{12600}-\frac{\pi ^ 2}{3840}. \qedhere
    \end{align*}
\end{proof}

\subsection{Expected length and variance of $k$th smallest segment}
Another common question is to ask for the expected length
of the shortest or longest segment.
The result for the problem was claimed in \cite[Problem 6.4.2]{orderstats}
without a proof or reference.
The same problem was considered in \cite{orderSE} and \cite{lesreid}.

\begin{proposition}
	[\cite{orderSE}]
	The $k$th smallest stick has expected length equal to
	\[ \frac{\frac{1}{k} + \frac{1}{k+1} + \dots + \frac 1n}{n}. \]
\end{proposition}
\begin{proof}
	Without loss of generality,
	assume the segment is broken into segments of
	length $s_0 \ge s_2 \ge \dots \ge s_n$, in that order.
	We are given that $s_1 + \dots + s_n = 1$,
	and want to find the expected value of each $s_k$.

	Set $s_i = x_i + \dots + x_n$ for each $i = 1, \dots, n$, where $s_0 = 0$.
	Then, we have $x_1 + 2x_2 + \dots + nx_n = 1$,
	and want to find the expected value of $s_k = x_k + \dots + x_n$.

	However, if we set $y_i = i x_i$,
	then we have $y_1 + \dots + y_n = 1$,
	so by symmetry $\mathbb{E}[y_i] = \tfrac{1}{n}$ for all $i$.
	Thus, $\mathbb{E}[x_i] = \tfrac{1}{in}$ for each $i$,
	and now
	\[\mathbb{E}[s_k] = \mathbb{E}[x_k] + \dots + \mathbb{E}[x_n]
		= \tfrac{1}{n}(\tfrac{1}{k} + \dots + \tfrac{1}{n})\]
	as desired.
\end{proof}

\begin{theorem}
	The variance of the length of the $k$th longest stick is equal to
	\[ \sum_{i = k}^{n} \frac{1}{i^2} \frac{1}{n(n + 1)} - \frac{1}{n^2(n + 1)}\left(\frac{1}{k} + \dots + \frac{1}{n}\right)^2. \]
\end{theorem}
\begin{proof}
	We have
	\begin{align*}
	\operatorname{Var}[s_k] & = \operatorname{Var}[x_k + \dots + x_n]\\
	& = \sum_{k \le i \le n} \operatorname{Var}[x_i] + 2 \sum_{k \le i < j \le n} \mathbb{E}[x_i x_j] - \mathbb{E}[x_i] \mathbb{E}[x_j]\\
	& = \sum_{k \le i \le n} \frac{1}{i^2} (\mathbb{E}[y_i^2] - \mathbb{E}[y_i]^2) + 2 \sum_{k \le i < j \le n} \frac{1}{ij} (\mathbb{E}[y_i y_j] - \mathbb{E}[y_i] \mathbb{E}[y_j])\\
	& = \sum_{k \le i \le n} \frac{1}{i^2} \mathbb{E}[y_i^2] + 2 \sum_{k \le i < j \le n} \frac{1}{ij} \mathbb{E}[y_i y_j] - \frac{1}{n^2}\left(\frac{1}{k} + \dots + \frac{1}{n}\right)^2
	\end{align*}
	by a well-known property of variance.

	Now, by symmetry, it is enough to calculate $\mathbb{E}[y_1^2]$ and $\mathbb{E}[y_1y_2]$.

	We now calculate $\mathbb{E}[y_1^2]$. This is equal to $\int_0^1 x^2 f(x) \; dx$, where $f$ is the probability density function of $y_1$. We can show that this pdf is $f(x) = (n - 1)(1 - x)^{n - 2}$, so now $\mathbb{E}[y_1^2] = \int_0^1 (n - 1) x^2 (1 - x)^{n - 2} \; dx = \tfrac{(n - 1)2!(n - 2)!}{(n + 1)!} = \tfrac{2}{n(n + 1)}$.

	Finally, to calculate $\mathbb{E}[y_1y_2]$, observe that
	\[1 = \mathbb{E}[(y_1 + \dots + y_n)^2] = n\mathbb{E}[y_1^2] + n(n - 1)\mathbb{E}[y_1y_2] = \tfrac{2}{n + 1} + n(n - 1)\mathbb{E}[y_1y_2]\]
	which gives $\mathbb{E}[y_1y_2] = \tfrac{1}{n(n + 1)}$.

	Thus, the variance of $s_k$ is equal to
	\begin{eqnarray*}
	& & \sum_{k \le i \le n} \frac{1}{i^2} \frac{2}{n(n + 1)} + 2 \sum_{k \le i < j \le n} \frac{1}{ij} \frac{1}{n(n + 1)} - \frac{1}{n^2}\left(\frac{1}{k} + \dots + \frac{1}{n}\right)^2\\
	& = & \sum_{i = k}^{n} \frac{1}{i^2} \frac{1}{n(n + 1)} + \frac{1}{n(n + 1)}\left(\frac{1}{k} + \dots + \frac{1}{n}\right)^2 - \frac{1}{n^2}\left(\frac{1}{k} + \dots + \frac{1}{n}\right)^2\\
	& = & \sum_{i = k}^{n} \frac{1}{i^2} \frac{1}{n(n + 1)} - \frac{1}{n^2(n + 1)}\left(\frac{1}{k} + \dots + \frac{1}{n}\right)^2. \qedhere
	\end{eqnarray*}
\end{proof}

\section{Breaking a square}
\label{sec:square}
Moving on from spaghetti noodles,
we considered breaking a square $[0,1]^2$ randomly into parts.
(Perhaps this corresponds to a saltine cracker.)
In this case, it is not clear what is meant by a ``random line''
through the square anymore, and we investigated two directions.

\subsection{Breaking through a random line through the center}
An easier case to imagine is to divide the square
by drawing a random line through the center.
\begin{proposition}
	Consider in this section the distribution in which
	we draw a random line through the center. Then:
	\begin{enumerate}[(a)]
		\item There are always $2n$ regions.
		\item The expected number of triangles is
			\[  2n-4 + \left( \frac{1}{2} \right) ^{n-2}. \]
		\item For $n=1$ both regions have area $\frac12$,
		and when $n > 1$, with probability $1$ no region
		has area $\ge \frac12$.
	\end{enumerate}
\end{proposition}
\begin{proof}
	\begin{enumerate}[(a)]
		\item This is $2n$, as there are $2n$ angles at the center, and one region for each angle. The probability that two lines coincide is zero, as there is an infinite number of lines to choose from.
		\item Either there are $2n-4$ or $2n-2$ triangles
		(the latter case occurring if all lines hit 
		the same opposite pairs of sides).

		The probability of having $2n-2$ regions be triangles is $2\left( \frac{1}{2} \right)^{n} = \left( \frac{1}{2} \right)^{n-1}$.
		So the probability of having $2n-4$ regions is $1- \left( \frac{1}{2} \right)^{n-1}$.
		Thus, our expected value is $(2n-2) \left( \frac{1}{2} \right)^{n-1} + (2n-4) (1-\left( \frac{1}{2}\right) ^{n-1}) = 2n-4 + \left( \frac{1}{2} \right) ^{n-2}$.

		\item If $n=1,$ since all lines passing through the center of the square split the square in half, this probability is 1.
If $n>1,$ since the square was already split in half by the first line, each piece will be guaranteed to be less than $\frac 1 2,$ as each piece is further split by the additional lines.
	\end{enumerate}
\end{proof}

\begin{proposition}
	Assume $n \ge 3$.
	Among the $2n$ regions formed when we draw $n$ random lines through the center
	the expected area of a randomly selected triangular region is
	\[ \frac{n^3+2^{n+1}(2n^2-n+5)-9n-16}{(n+2)(n+1)(n-1)}. \]
\end{proposition}
\begin{proof}
	Note that all the lines must pass through opposite sides of the square. Let us do casework on the number of lines that pass through the side $x=0$. 

	There is a probability of $\frac{\binom{n}{k}}{2^n}$ that exactly $k$ lines intersect the side $x=0$.

	We have two cases.

	\begin{itemize}
	\item \textbf{Case 1}: There does not exist a side that is not intersected by a line.

	Each line intersects exactly one of the two sides: $x=0, y=0$. Let the $k$ lines intersect $x=0$ at $(0,a_1), (0,a_2), \ldots (0,a_k)$, where $a_1\le a_2\le \ldots \le a_k$. Similarly, let the other $n-k$ lines intersect $y=0$ at $(b_1,0),(b_2,0),\ldots , (b_{n-k},0)$, where $b_1\le b_2\le \ldots \le b_{n-k}$. Consider the quadrilateral with vertices $(0,0),(0,a_1),(\frac{1}{2},\frac{1}{2}),(b_1,0)$. Call such a quadrilateral \emph{cornered}.
	The area of this quadrilateral is $\frac{a_1+b_1}{4}$.
	The area of the region that is not occupied by a triangle within the square is composed of 4 cornered quadrilaterals for this case.

	Let $x_i=a_i-a_{i-1}$, for $2\le i \le k$, let $x_1=a_1$, and let $x_{k+1}=1-a_k$. Since $x_1+x_2+\ldots + x_{k+1} = 1$, by symmetry the expected value of $x_1$ is $\frac{1}{k+1}$. So the expected value of $a_1$ is also $\frac{1}{k+1}$. Similarly, the expected value of $b_1$ is $\frac{1}{n-k+1}$.

	Hence, our expected value for $\frac{a_1+b_1}{4}$ is
	\begin{align*}
	\frac{1}{4}\sum\limits_{k=1}^{n-1} \frac{\binom{n}{k}}{2^n} \left( \frac{1}{k+1} + \frac{1}{n-k+1} \right)
	&= \frac{1}{2^{n+2}}\sum\limits_{k=1}^{n-1} \binom{n}{k} \left( \frac{1}{k+1} + \frac{1}{n-k+1} \right) \\
	&= \frac{1}{2^{n+2}}\sum\limits_{k=1}^{n-1} \frac{1}{n+1}\binom{n+1}{k+1} + \frac{1}{n+1}\binom{n+1}{k} \\
	&= \frac{1}{2^{n+2}(n+1)}\sum\limits_{k=1}^{n-1} \binom{n+2}{k+1} \\
	&= \frac{2^{n+2}-1-(n+2)-\binom{n+2}{n}-\binom{n+2}{n+1}-1}{2^{n+2}(n+1)} \\
	&= \frac{1}{n+1}\left(1 - \frac{n^2+3n+4}{2^{n+3}} \right).
	\end{align*}

	To account for all $4$ cornered quadrilaterals,
	we multiply this count to get $\frac{4}{n+1}\left(1 - \frac{n^2+3n+4}{2^{n+3}} \right)$.

	Since there are $2n-4$ triangles, the expected area of a triangle in this case is 
	\[ \frac{1-\frac{4}{n+1}\left(1 - \frac{n^2+3n+4}{2^{n+3}} \right)}{2n-4}=\frac{n^2+3n+2^{n+1}(n-3)+4}{2^{n+2}(n+1)(n+2)}. \]

	\item \textbf{Case 2}: The cornered quadrilaterals do not exist.

	So $k=0$ or $n$. Without the loss of generality, suppose no line intersects the side $y=0$. Let the $n$ lines intersect the side $x=0$ at points $(0,a_1), (0,a_2), \ldots (0,a_n)$. Then the area of the region not occupied by triangles consists of two pentagons. Let us find the expected area of the pentagon that contains the side $y=0$. Its vertices are $(0,0), (0,a_1), (\frac{1}{2}, \frac{1}{2}), (1,1-a_n), (1,0)$. Its area is equal to $\frac{a_1}{4} + \frac{1-a_n}{4} + \frac{1}{4}$.

	Similar to above, the expected value of $a_1$ and $1-a_n$ are both $\frac{1}{n+1}$. Hence, the expected area of the pentagon is $\frac{1}{2(n+1)}+\frac{1}{4}$. To account for both pentagons, we multiply this count by $2$ to get $\frac{1}{n+1} + \frac{1}{2}$.

	The probability of this happening is $\frac{2}{2^n} = \frac{1}{2^{n-1}}$, so the expected contribution for this case is $\frac{1}{2^{n-1}}\left( \frac{1}{n+1} + \frac{1}{2} \right)$. 

	Since there are $2n-2$ triangles, the expected area of a triangle is $\frac{1-\frac{1}{2^{n-1}}\left( \frac{1}{n+1} + \frac{1}{2} \right)}{2n-2} = \frac{2^n(n+1)-n-3}{2^{n+1}(n+1)(n-1)}$.
	\end{itemize}
	Combining both cases, the expected area of a triangle is 
	\[ \frac{n^3+2^{n+1}(2n^2-n+5)-9n-16}{(n+2)(n+1)(n-1)}. \qedhere \]
\end{proof}

\subsection{Randomly join two points on the perimeter}
A harder variant is to imagine the dividing line is formed
by taking two randomly selected points on the perimeter of $[0,1]^2$,
and drawing the line joining them.
(In particular, there is $1/4$ chance that the dividing line
is along a side of the square, and thus has no effect.)
In this case, it is still possible to compute
the expected number of regions.

\begin{proposition}
The expected number of regions is
\[ \frac{17}{64} \binom{n}{2} + \frac34 n + 1. \]
\end{proposition}
\begin{proof}
	We mark all $2n$ points which are endpoints of a square,
	and draw in all the dividing segments not contained
	in the side of any square.
	Then, mark the intersections of any two dividing segments.
	We can then consider a planar graph
	whose vertices are all the marked points
	(including the $2n+4$ on the boundary of the square)
	and whose edges are all the drawn segments
	between pairs of vertices
	(including the $2n+4$ segments forming the boundary of the square).
	We will use Euler's formula $V+F-E=2$,
	where $V$, $E$, and $F$ are the expected number
	of vertices, edges, and faces, respectively.

	We begin by computing the expected number of vertices. There are $4$ vertices from the square and $2n$ vertices from the endpoints of each line (the probability of two lines sharing the same vertex is zero).

	Now, we find the expected number of vertices inside of the square. By linearity of expectation, this quantity is equal to the probability of two lines intersecting times the total number of pairs of lines. There are $\binom{n}{2}$ pairs of lines. Given the endpoints of two lines, the probability that they could intersect inside the square is equal to the probability that no three endpoints are collinear (on a side of the square).

	To compute this, we apply complementary counting and find the probability that at least 3 endpoints are on one side. We have two cases. There is a $\binom{4}{3} \left( \frac{1}{4} \right)^2 \left( \frac{3}{4} \right) = \frac{3}{16}$ probability of having 3 points on one side of the square and 1 point is on the another. There is a probability of $\left( \frac{1}{4} \right)^3 = \frac{1}{64}$ of having all 4 points on the same side. So the probability of it being possible to make the lines intersect is $1 - ( \frac{3}{16} + \frac{1}{64} ) = \frac{51}{64}$. The probability that the lines are the diagonals of the quadrilateral is $\frac{1}{3}$, because when we take a vertex of the quadrilateral and connect it to another vertex, there are 3 other vertices, but only 1 is good. Hence, the probability that a pair of given lines intersect is $\frac{51}{64} \cdot \frac{1}{3} = \frac{17}{64}$. So the expected number of vertices is $\frac{17}{64} \binom{n}{2} + 2n + 4$.

	Let us now compute the expected number of edges. We start out by finding the number of edges that are on the boundary of the square. For each point we add to the perimeter, we add one more edge. Initially, we had 4 sides. If we have $n$ lines, then we have $2n$ points. So the number of edges that are on the boundary is $2n+4$.

	Now we compute the expected number of edges inside the square. For each pair of lines that intersect inside the square, they add $2$ new edges. The expected number of edges we start out with inside of the square is $\frac{3n}{4}$ (because $n$ lines). The expected number of intersection points inside of the square is $\frac{17}{64}\binom{n}{2}$. Therefore, the expected number of edges inside the square is $\frac{3n}{4}+ 2 \cdot \frac{17}{64} \binom{n}{2} = \frac{3n}{4} + \frac{17}{32} \binom{n}{2}$. 

	Adding both cases up, we have an expected total of $\frac{17}{32} \binom{n}{2} + \frac{11n}{4}+4$ edges.

		Plugging $V=\frac{17}{64} \binom{n}{2} + 2n + 4$ and $E=\frac{17}{32} \binom{n}{2} + \frac{11n}{4} +4$ into $V+F-E=2$, we get that $F=\frac{17}{64} \binom{n}{2} + \frac{3n}{4} +2$. However, we need to subtract out 1 for the face outside the square. Thus, our desired answer is $\frac{17}{64} \binom{n}{2} + \frac{3n}{4} + 1$.
\end{proof}

A possible direction to continue is to investigate:
\begin{ques}
	What is the expected area of the largest piece
	formed by this procedure?
\end{ques}
\begin{remark}
If $n=1$, the expected value is $\frac{5}{6}$.
There is a $\frac{2}{3}$ probability that the line forms a triangle and a pentagon.
The expected area of the triangle is $\frac{1}{8}$,
so the expected area of the largest region is $\frac{7}{8}$.
There is a $\frac{1}{3}$ probability that the line forms two trapezoids.
Let the two bases of the trapezoid have lengths $a_1$ and $b_1$.
Then, we want to find the expected value of the
$\frac{\max(a_1+b_1,(1-a_1)+(1-b_1))}{2}
	= \frac{\max(a_1+b_1,2-(a_1-b_1))}{2} = \frac{3}{4}$.
This gives $\frac{2}{3} \times \frac{7}{8}
+ \frac{1}{3} + \frac{3}{4} = \frac{5}{6}$.
\end{remark}

One can find some exponential upper and lower bounds in this sense.
\begin{proposition}
	Let $\lambda \approx 0.345$ be the probability
	that a randomly selected line does not
	intersect the disk centered at the center of the square,
	with radius $(2\pi)^{-\half}$.
	Then the probability some piece has area at least $\half$
	after $n$ random lines is at least $\lambda^n$.
\end{proposition}

\begin{proposition}
	The probability that some piece has area at least $\half$
	after $n$ random lines is at most $3(11/12)^{n/2}$.
\end{proposition}
\begin{proof}
	Mark four parts of the square as shown below.
	\begin{center}
		\includegraphics{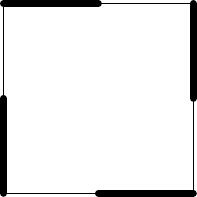}
	\end{center}
	The probability that some segment joins the two horizontal bold half-sides
	is $\frac{1}{12}$; likewise for the vertical bold half-sides.
	However, if two such segments are drawn, then
	certainly no piece can have area exceeding $\half$.

	Now, observe that
	\begin{itemize}
		\item The probability that some segment
		joining the horizontal bold half-sides
		is drawn in the first $\left\lfloor n/2 \right\rfloor$ cuts is at least
		$\alpha := 1 - (\frac{11}{12})^{\left\lfloor n/2 \right\rfloor}$.
		\item The probability that some segment
		joining the vertical bold half-sides
		is drawn in the last $\left\lceil n/2 \right\rceil$ cuts is at least
		$\beta := 1 - (\frac{11}{12})^{\left\lceil n/2 \right\rceil}$.
	\end{itemize}
	This gives an overall upper bound of
	\begin{align*}
		1 - \alpha\beta
		&= 1 - \left[ 1 - \left( \frac{11}{12} \right)^{\left\lfloor n/2 \right\rfloor} \right]
		\left[ 1 - \left( \frac{11}{12} \right)^{\left\lceil n/2 \right\rceil} \right] \\
		&= \left( \frac{11}{12} \right)^{\left\lfloor n/2 \right\rfloor}
		+ \left( \frac{11}{12} \right)^{\left\lceil n/2 \right\rceil}
		- \left( \frac{11}{12} \right)^n < 3\left( \frac{11}{12} \right)^{n/2}. \qedhere
	\end{align*}
\end{proof}

\section{The number of $k$-gons}
\label{sec:numtri}

This chapter addresses a combinatorial generalization:
for any $3 \le k \le n$ and $0 \le m \le \binom nk$
let $P(k,n,m)$ be the probability of being able to form $m$ or more $k$-gons
after breaking the stick into $n$ pieces.

\subsection{Some known results}
The original broken stick problem is $P(3,3,1) = \frac14$
in this notation.
In addition, in \cite{noodle} and \cite{illinois}
\begin{align*}
	P(n,n,1) &= 1 - \frac{n}{2^{n-1}} \\ 
	P(3,n,1) &= 1 - \prod_{k=2}^n k(F_{k+2}-1)^{-1} \\ 
	P\left(3,n,\binom n3\right) &= \frac{1}{\binom{2n-2}{n}}. 
\end{align*}
Here $F_k$ denotes the $k$th Fibonacci number.

It would be interesting to compute $P(k,n,m)$ for more triples.
We give a few basic results below.

\subsection{Nonvanishing of $P(k,n,m)$}
We also have the following nonvanishing result.
\begin{proposition}
	Suppose $[0,1]$ is broken randomly into $n$ pieces.
	The probability of being able to form exactly $m$ $k$-gons is nonzero
	for all $0 \le m \le \binom nk$.
	In other words, $P(k,n,0) > 0$ and $P(k,n,m) - P(k,n,m-1) > 0$
	whenever $0 < m \le \binom{n}{k}$.
\end{proposition}
\begin{proof}
	Fix $k$.  Let us prove by induction on $n$ that all values $m = 0, \dots, \binom nk$
	are achievable.

	The base case of $n=k$ is true. In fact, we know that $P(n,n,1)=1-\frac{n}{2^{n-1}}$.

	For the inductive step, let us assume that for some positive integer $n=j_1$ all values from $m=0$ to $m=\binom{j_1}{k}$ are possible. We show that all values from $m=0$ to $m=\binom{j_1+1}{k}$ are possible for $n=j_1+1$.

	Note that it is possible to make the largest segment sufficiently large such that it will be impossible to form a $k$-gon with that segment, so we can ignore that segment. Applying our inductive hypothesis, on the other $j_1$ pieces, all values from $0$ to $\binom{j_1}{k}$ are possible.

	We now need to show that all values between $\binom{j_1}{k}+1$ and $\binom{j_1+1}{k}$ are possible. We want to break our stick into pieces with lengths $a_1 \le a_2 \le \ldots \le a_{j_1}$, such that any $k$ pieces form a triangle and the set $S=\{ a_{i_1} + a_{i_2} + \ldots a_{i_{k-1}} | 1 \le i_1<i_2< \ldots < i_{k-1} \le j_1 \}$ to have $\binom{j_1}{k-1}$ distinct values. Let these values be $x_1< x_2 < \ldots  < x_{\binom{j_1}{k-1}}$. For $\binom{j_1}{k-1}+l$ to be possible for some positive integer $0\le l\le \binom{j_1}{k-1}$, we can let $x_l<a_{j_1+1}<x_{l+1}$, where $x_0=0$ and $x_{\binom{j_1}{k-1}+1} = \infty$.

	We can construct such a sequence with induction. The base case is true. For the inductive step, assume that it is possible to construct such a sequence for $n=j_2 \ge k$. To show that it is possible for $j_2+1$, we can let $a_{j_2}<a_{j_2+1}<a_1+a_2+ \ldots + a_{k-1}$. So now, there are $\binom{j_2}{k}+\binom{j_2}{k-1} = \binom{j_2+1}{k}$ distinct values in $S$. Also, it is possible to avoid repeated sums as there are an infinite number of possible values $a_{j_2+1}$ can take on but only $\binom{j_2}{k-1}$ that must be avoided.

	So indeed, all values from $0$ to $\binom{n}{k}$ are possible.
\end{proof}

\subsection{An application of $P(3,4,1) = 4/7$}
Suppose we repeatedly select uniformly random breaking points
until any three segments form a triangle.
Let $X$ be the number of times the segment is broken.
We will compute $\mathbb P[X=3]$.

\begin{theorem}
	\label{thm:talkon_main}
	We have $\mathbb P[X=3] = \frac{39}{112}$.
\end{theorem}

From \cite{illinois} we know $P(3,4,1) = 4/7$.
We need one more lemma.

\begin{lemma}
	Suppose we repeatedly select uniformly random breaking points
	until any three segments form a triangle.
	Let $E$ be the event that
	\begin{enumerate}
		\item[(i)] After two breaks, we have a triangle.
		\item[(ii)] After the third break, we have no triangle.
	\end{enumerate}
	Then $\mathbb P[E]=\frac{3}{112}$.
\end{lemma}
\begin{proof}
The cases (that our last breaks are at $x$ or $y$ or $z$ from breakpoints at $0<x<y<z<1$) is essentially the same so we consider when our last break is at $x$ first. If our final point is $(a,b,c,d)$ then before the last break the segments has length $a+b,c,d$.

By our assumption this creates a triangle so $\max\{a+b,c,d\}\leq \frac{1}{2}$. Now the region satisfying the previous equation is exactly half the octahedron in part 2 above. (It is a square pyramid with base $M_{AC}M_{AD}M_{BD}M_{BC}$ and top vertex $M_{CD}$.)

We need to also make sure that $(a,b,c,d)$ cannot create a triangle, that is $(a,b,c,d)\not\in \mathcal{R}$. With these two constraints, $(a,b,c,d)$ must be in one of six ``small tetrahedrons'' from part 2 above. These have total volume $\frac{6}{224}=\frac{3}{112}$.

Hence this is the probability that $(a,b,c,d)$ does not create a triangle, but $a+b,c,d$ creates one. Now we also need our last break to be $x$, the least-valued point; this has probability $\frac{1}{3}$.
Therefore, the probability of all these happening:
\begin{itemize}
	\item After two breaks, we have a triangle.
	\item After the third break, we do not have a triangle
	\item The third break happened at $x$, the least breakpoint.
\end{itemize}
is exactly $\frac{1}{3}\cdot\frac{3}{112}=\frac{1}{112}$.

Now, our third break can also be at $y$ or $z$ so $\mathbb P[E]$ is three times $\frac{1}{112}$, that is, $\frac{3}{112}$.
\end{proof}

\begin{proof}
	[Proof of Theorem~\ref{thm:talkon_main}]
Consider three events:
\begin{itemize}
\item $C_1$: After two breaks, we have a triangle. After the third break, we still have a triangle.
\item $C_2$: After two breaks, we have a triangle. After the third break, we do not have a triangle.
\item $C_3$: After the third break, we do not have a triangle. After two breaks, we have a triangle.
\end{itemize}
We can see that
\begin{align*}
\mathbb P[X=2] &= \mathbb P[C_1]+\mathbb P[C_2] \\
\mathbb P[X=3] &=\mathbb P[C_3] \\
\mathbb P[T] &= \mathbb P[C_1]+\mathbb P[C_3] \\
\mathbb P[E] &= \mathbb P[C_2].
\end{align*}
and hence 
\[ \mathbb P[X=3] = \mathbb P[T] + \mathbb P[E] - \mathbb P[X=2]
	= \frac{4}{7}+\frac{3}{112}-\frac{1}{4}
	= \frac{39}{112}. \qedhere \]
\end{proof}

\begin{corollary}
	We have $\mathbb E[X] \le 3.398$.
\end{corollary}
\begin{proof}
Let $T_n$ be the event that at least one triangle can be formed after $n$ breaks ($=n+1$ pieces).
If there is a triangle after $n$ breaks, we must get the first triangle after some $k\leq n$ breaks, so $\mathbb P[X\leq n] \geq  \mathbb P[T_n].$
From \cite{illinois} we have, when $n\geq 1$,
$\mathbb P[T_n] =1-\prod_{k=2}^{n+1} \frac{k}{F_{k+2}-1}$.
As $X$ must be a positive integer,
\begin{align*}
	\mathbb E[X] &= \sum_{n=1}^\infty n\cdot \mathbb P[X=n] \\
	&= \sum_{n=1}^\infty \mathbb P[X\geq n]
		= \sum_{n=1}^\infty \left(1-\mathbb P[X\leq n-1]\right) \\
	&= 1+\sum_{n=1}^\infty \left(1-\mathbb P[X\leq n]\right)\leq 1+ \sum_{n=1}^\infty \left(1-\mathbb P[T_n]\right) \\
	&=1+\sum_{n=2}^\infty \prod_{k=2}^{n} \frac{k}{F_{k+2}-1}.
\end{align*}
Here the last sum of products has a numerical value a little less than $2.424$,
giving a bound of $\mathbb E[X] \leq 3.424$.
	If we take into account the difference of $\mathbb P[X\leq 3]$
	and $\mathbb P[T_3]$ which is $\frac{3}{112}$,
	we can push the bound down further to
	$\mathbb E[X] \leq 3.424-\frac{3}{112} \approx 3.398$.
\end{proof}

We have a corresponding (very crude) lower bound:
\[ \mathbb E[X] > 2\cdot\mathbb P[X=2]+3\cdot\mathbb P[X=3]
	+4\cdot\mathbb P[X\geq 4] = \frac{353}{112}\approx 3.151. \]
For comparison, the empirical value of $\mathbb E[X]$ is about $3.351$.

\bibliographystyle{alpha}
\bibliography{refs}

\newcommand{\etalchar}[1]{$^{#1}$}
\begin{thebibliography}{KLT{\etalchar{+}}}

\bibitem[AN05]{spaghetti}
Basile Audoly and S\'ebastien Neukirch.
\newblock Fragmentation of rods by cascading cracks: Why spaghetti does not
  break in half.
\newblock {\em Phys. Rev. Lett.}, 95:095505, Aug 2005.

\bibitem[AOP]{aops}
AOPS.
\newblock Crowdmath 2017: The broken stick problem.
\newblock Art of {P}roblem {S}olving forums.
\newblock URL:\url{https://aops.com/polymath/mitprimes2017b/}.

\bibitem[DG]{noodle}
Carlos D'Andrea and Emiliano G\'{o}mez.
\newblock The broken spaghetti noodle.
\newblock
  URL:\url{https://atlas.mat.ub.edu/personals/dandrea/emiliano_gomez.ps}.

\bibitem[DN03]{orderstats}
Herbert~A. David and H.~N. Nagaraja.
\newblock {\em Order Statistics}.
\newblock Wiley-Interscience, 3 edition, 8 2003.

\bibitem[hl]{orig}
Michael~Lugo (https://mathoverflow.net/users/143/michael lugo).
\newblock If you break a stick at two points chosen uniformly, the probability
  the three resulting sticks form a triangle is 1/4. is there a nice proof of
  this?
\newblock MathOverflow.
\newblock URL:\url{https://mathoverflow.net/q/2014} (version: 2013-08-27).

\bibitem[KLT{\etalchar{+}}]{illinois}
Lingyi Kong, Luvsandondov Lkhamsurren, Abigail Turner, Aananya Uppal, and A.~J.
  Hildebrand.
\newblock Random points, broken sticks, and triangles project report.
\newblock
  URL:\url{http://www.math.illinois.edu/~ajh/ugresearch/brokenstick-spring2013report.pdf}.

\bibitem[Mat99]{gbook1}
A.M. Mathai.
\newblock {\em An Introduction to Geometrical Probability: Distributional
  Aspects with Applications (Statistical Distributions \& Models with
  Applications)}.
\newblock CRC Press, 1st edition, 12 1999.

\bibitem[Pet03]{maxquad}
Thomas Peter.
\newblock Maximizing the area of a quadrilateral.
\newblock {\em College Mathematics Journal}, September 2003.

\bibitem[PRI]{mitprimes}
MIT PRIMES.
\newblock Primes: Program for research in mathematics, engineering and science
  for high school students.
\newblock MIT {M}athematics.
\newblock URL:\url{http://math.mit.edu/research/highschool/primes/index.php}.

\bibitem[Rei]{lesreid}
Les Reid.
\newblock Advanced problem \#34.
\newblock URL:\url{http://people.missouristate.edu/lesreid/Adv34.html}.

\bibitem[Spi]{orderSE}
Mike Spivey.
\newblock If a 1 meter rope is cut at two uniformly randomly chosen points,
  what is the average length of the smallest piece?
\newblock Mathematics Stack Exchange.
\newblock URL:\url{http://math.stackexchange.com/q/13972} (version:
  2011-01-19).

\end{thebibliography}

\end{document}